\documentclass[11 pt]{amsart}
\usepackage{xypic}
\usepackage{amssymb}
\usepackage{cite}
\usepackage{cancel}

\usepackage[OT2,T1]{fontenc}
\DeclareSymbolFont{cyrletters}{OT2}{wncyr}{m}{n}
\DeclareMathSymbol{\sha}{\mathalpha}{cyrletters}{"58}

\newtheorem{theorem}{Theorem}[section]
\newtheorem{lemma}[theorem]{Lemma}
\newtheorem{prop}[theorem]{Proposition}
\newtheorem{cor}[theorem]{Corollary}
\newtheorem{defn}[theorem]{Definition}
\newtheorem{remark}[theorem]{Remark}
\newtheorem{conj}[theorem]{Conjecture}

\newcommand{\nc}{\newcommand}
\nc{\rnc}{\renewcommand}
\rnc{\P}{\mathbb P}
\nc{\R}{\mathbb R}
\nc{\C}{\mathbb C}
\nc{\A}{\mathcal A}
\nc{\E}{\mathcal E}
\nc{\Q}{\mathbb Q}
\nc{\Z}{\mathbb Z}
\rnc{\O}{\mathcal O}
\nc{\Sym}{\text{Sym}}
\nc{\Spec}{\text{Spec}\,}
\nc{\eps}{\epsilon}
\nc{\Pic}{\text{Pic}}
\nc{\ov}{\overline}
\nc{\X}{\mathcal X}
\nc{\F}{\mathbb F}
\nc{\G}{\mathbb G}
\nc{\D}{\mathfrak D}
\nc{\PP}{\mathfrak P}
\nc{\T}{\mathfrak T}
\rnc{\H}{\mathcal H}
\nc{\LL}{\mathcal L}
\nc{\Ra}{\Rightarrow}
\nc{\M}{\mathcal M}
\nc{\W}{\mathcal W}
\nc{\B}{\mathcal B}

\rnc{\F}{\mathcal F}

\nc{\la}{\langle}
\nc{\ra}{\rangle}

\nc{\Bl}{\text{Bl}}
\nc{\codim}{\text{codim}}
\nc{\im}{\text{im}}

\nc{\us}{\underset}
\nc{\os}{\overset}
\nc{\ul}{\underline}

\rnc{\L}{\mathbb L}
\rnc{\S}{\mathcal S}

\nc{\NS}{\text{NS}}
\nc{\MW}{\text{MW}}
\nc{\NL}{\text{NL}}
\nc{\End}{\text{End}}
\nc{\Hom}{\text{Hom}}
\nc{\Ext}{\text{Ext}}
\nc{\Def}{\text{Def}}
\nc{\Jac}{\text{Jac}}
\nc{\bs}{\backslash}

\nc{\sslash}{\mathbin{/\mkern-6mu/}}

\title{Counting Curves on a Weierstrass model}
\author{Fran\c{c}ois Greer}

\begin{document}

\begin{abstract}Let $X\to \mathbb P^2$ be the elliptic Calabi-Yau threefold given by a general Weierstrass equation.  We answer the enumerative question of how many discrete rational curves lie over lines in the base, proving part of a conjecture by Huang, Katz, and Klemm.  The key inputs are a modularity theorem of Kudla and Millson for locally symmetric spaces of orthogonal type and the deformation theory of $A_n$ singularities.\end{abstract}

\maketitle

\section{Introduction}
\noindent
We present a strategy for proving modularity of curve counts when $X$ is an elliptic Calabi-Yau threefold.  This work was inspired by the physics conjectures of Huang-Katz-Klemm in \cite{hkk}, and by the modular form techniques of Maulik-Pandharipande in \cite{mp}.  
\begin{defn} A threefold $X$ is {\it elliptic} if it admits a morphism $\pi:X\to B$ whose generic fiber is a smooth genus 1 curve, and a section $z:B\to X$.\end{defn}
\noindent
The following conjecture is now well established in the literature.
\begin{conj}  If $X$ is an elliptically fibered Calabi-Yau variety, then its Gromov-Witten invariants are governed by (quasi-)modular forms.\end{conj}
\begin{remark} In dimension 1, the Calabi-Yau varieties are elliptic curves, and in dimension 2, the Calabi-Yau varieties are K3 surfaces, which can always be deformed to elliptic fibrations over $\P^1$.  By deformation invariance, the elliptic condition only affects dimensions 3 and higher.\end{remark}
\noindent
In the 1990's, Dijkgraaf used mirror symmetry heuristics to conjecture the quasi-modularity of square tiled surface counts \cite{Dijk}.  This was subsequently proven by Okounkov and his collaborators, using representation theory of the Virasoro algebra.  Katz-Klemm-Vafa \cite{kkvblack} used M-theory to conjecture a modular product formula for Gromov-Witten invariants of the K3 surface (KKV formula).  Yau and Zaslow \cite{yz} proved a beautiful formula counting {\it rational} curves on K3 surfaces using the relative compactified Jacobian.  The formula for primitive classes was proven in all genera by Bryan-Leung \cite{bryanleung}, and then completed by Pandharipande-Thomas \cite{kkvpt} for all classes.  The present work addresses the modularity in dimension 3.  A general conjecture for the Weierstrass model over $\P^2$ is presented in \cite{hkk}, using the topological string partition function.  Significant progress in this direction has been made by Oberdieck-Shen in \cite{oberdieck}, using the sheaf counting theories.  Here, we approach the problem from the Gromov-Witten side, and focus on the question of counting rational curves which lie over lines in the base.\\\\
In this paper, we take $X$ to be a general Weierstrass fibration $\pi:X \to \P^2$.  The Weierstrass equation
$$zy^2 = x^3+ a_{12}xz^2+b_{18}z^3$$
cuts $X$ out of the projective bundle $\P(\O(6)\oplus \O(9)\oplus \O)$ over $\P^2$, where 
$$a_{12}\in H^0(\P^2,\O(12)),\,b_{18}\in H^0(\P^2,\O(18))$$
are general forms of the appropriate degrees.  The fibration has a section with image 
$$[x:y:z]=[0:1:0].$$ 
The discriminant form $\Delta = 4a_{12}^3 + 27b_{18}^2$ cuts out a curve in $\P^2$ over which the fibers become singular.  For general $X$, only cuspidal and nodal cubics appear as fibers.\\\\ 
The same geometry can also be realized as a hypersurface 
$$X_{18}\subset \P(1,1,1,6,9)$$
in weighted projective space, blown up at its unique singular point $X_{18}\cap \P(6,9)$.  The map to $\P^2$ is projection away from $\P(6,9)$, and the exceptional divisor of the blow up is the image of the zero section.  The second integral homology group is
$$H_2(X) = \Z \ell \oplus \Z f,$$
where $\ell$ is the class of a line in $\P^2$ pushed through the section, and $f$ is the class of a fiber.\\\\
Our goal is to count curves on $X$ in the class $\beta= d\ell + n f$.  We let $\ov{M}_0(X,\beta)$ be the moduli stack of stable maps $u:C\to X$, with $C$ a nodal rational curve and $|Aut(u)|<\infty$, such that $u_*[C]=\beta$.  The Gromov-Witten invariants of $X$ are virtual counts of stable maps to $X$, given by the integral
$$N^X_{0,\beta} = \int_{[\ov{M}_0(X,\beta)]^{vir}} 1.$$
Composing with the map $\pi$ and then stabilizing, we have a morphism
$$\pi_*:  \ov{M}_0(X,\beta) \to \ov{M}_0(\P^2, d)$$
whose fiber of $\pi_*$ over a stable map $v:D\to \P^2$ is a stack of stable maps to the elliptic surface
$$S:=D\times_{\P^2}X \to D,$$
whose class must push forward to $\beta\in H_2(X)$.  To avoid bad singularities in $S$, we restrict to the case $d=1$, so that $S=\pi^{-1}(L)$ has at worst canonical (ADE) singularities.  The domain curve $C$ has a unique component which dominates $L\simeq \P^1$, and that component is a section of the fibration $S\to \P^1$.  All the surfaces obtained this way have a zero section, but some of them have additional sections as well.  The presence of extra sections can be detected using Hodge structure of the surface $S$.\\\\
The moduli space $\ov{M}_0(X,\ell+nf)$ breaks up into closed-open components
$$\ov{M}_0(X,\ell+nf) = W_{n,0} \sqcup W_{n,1} \sqcup \dots \sqcup W_{n,n},$$
where $W_{n,i}$ consists of stable maps whose image consists of a smooth horizontal (section) component of class $\ell+if$, glued with $(n-i)$ fiber components.  To compute the Gromov-Witten invariant $N^X_{0,\ell+nf}$, we first consider the contribution of the smooth component $W_{n,n}$:
$$h(n) := \int_{[W_{n,n}]^{vir}} 1.$$
The case $h(0)=-3$ was computed by Yau and Zaslow in their study of local $\P^2$.  We compute the general $h(n)$ using a modularity result from Noether-Lefschetz theory.  
\begin{theorem}
For $n\geq 1$, $h(n)$ is the coefficient of $q^n$ in the series
$$ \frac{31}{48}E_4(q)^4+\frac{113}{48} E_4(q)E_6(q)^2+93582\,\Theta_1(q) - 46008\, \Theta_1(q)^2 - 324\,\Theta_2(q),$$
where $\Theta_i(q)$ denotes the root lattice theta function for $A_i$.
\end{theorem}
\begin{remark} The components $W_{n,1}$ and $W_{n,2}$ are empty (see Corollary ~\ref{empty}), so in particular 
$$h(1)=h(2)=0.$$
\end{remark}
\noindent
The contributions of the remaining $W_{n,i}$ is obtained by induction, combining a deformation argument with a result of Bryan-Leung counting nodal rational curves on a K3 surface, which generalizes to other elliptic surfaces with the appropriate reduced virtual class.  Writing 
$$c_{n,0}=\int_{[W_{n,0}]^{vir}} 1,$$
the full Gromov-Witten invariant is given by
$$N^X_{0,\ell+nf} = c_{n,0}+\sum_{i=1}^n   h(i)[\eta(q)^{-36}]_{n-i},$$
where $\eta(q) = \prod_{m=1}^{\infty} (1-q^m) $.  The contribution $c_{n,0}$ is significantly harder to compute because the geometry is not localized to isolated lines in $\P^2$.  Comparing with the ansatz in Section 4.6  of \cite{hkk}, we expect the Gromov-Witten generating series to be product of a classical modular form and $\eta(q)^{-36}$.  To match this expectation, we leave the:
\begin{conj}
$$c_{n,0}= 3 [\eta(q)^{-36}]_n + \sum_{i=1}^n [\Theta(q)]_i[\eta(q)]_{n-i}.$$\\\\
\end{conj}
\noindent
{\bf Ackowledgements.}  The author would like to thank Jim Bryan for suggesting the problem, and Jun Li for constant encouragement and numerous ideas.  He also benefitted from conversations with Martijn Kool, Zhiyuan Li, Davesh Maulik, and Georg Oberdieck.\\\\\\

\section{Elliptic Surfaces}
\noindent
If $[u:C\to X]\in W_{n,n}$, then $\pi\circ u:C\to \P^2$ is an isomorphism onto a line $L\subset \P^2$.  The curve $C$ embeds into the surface $S=\pi^{-1}(L)$ as a section of the fibration, which can be recovered from the N\'{e}ron-Severi lattice of $S$.\\\\
If $L$ meets the discriminant curve $\Delta\subset \P^2$ transversely at nonsingular points, then $S$ will be a regular elliptic surface over $\P^1$.  Here, we collect some standard facts about this type of elliptic surface, following \cite{mir}.  Such surfaces are sometimes called type $E(3)$.  They have 36 nodal fibers, and their Hodge numbers are
\begin{align*}
h^{0,1}(S)&=0\\
h^{0,2}(S)&=2\\
h^{1,1}(S)&=30.
\end{align*}
We have $(R^1\pi_* \O_S)^\vee\simeq \O_L(3)$, so the canonical bundle is given by
$$K_S \simeq \pi^*(\O_L(3)\otimes K_L) \simeq \pi^* \O_L(1).$$
The Picard rank $\rho(S)\geq 2$, since $\NS(S)$ contains the fiber class $f$ and the zero section class $z$.  Any section class has self-intersection $-3$, by the adjunction formula:
$$-2 = 2g-2 = K_S\cdot C + C\cdot C = 1 + C\cdot C.$$
Thus, $\NS(S)$ will always contain the lattice
$$\langle f,z\rangle = \begin{pmatrix} 0 & 1 \\ 1 & -3 \end{pmatrix}.$$
We refer to this sublattice as the polarization, since it comes to us for free, with a marked basis.
\begin{remark} \label{canfib} Unlike the K3 case, the fibration structure on $S$ is canonical.  The linear system $|K_S|=|f|$ gives the fibration $S\to \P^1$. \end{remark}
\noindent
Noether-Lefschetz theory is the study of Picard rank jumping in families of algebraic surfaces.  The very general surface will have some Picard rank, and Hodge theory suggests that the rank will jump along (countably many) loci of codimension $p_g$ in the moduli space.  To understand possible sources of jumping in our situation, we record the short exact sequence of Shioda-Tate, which describes the N\'{e}ron-Severi group of $S$ in terms of the fibration structure.
$$0\to A(S)\to \NS(S) \to \MW(S) \to 0.$$
Here, $A(S)$ is the sublattice spanned by the zero section $z$ and all vertical classes, and $\MW(S)$ is the Mordell-Weil group of the generic fiber, which is an elliptic curve over $k(\P^1)$.  The group $\MW(S)$ is torsion-free in our case (see Appendix A), so the intersection form on $\NS(S)$ splits the sequence.  In particular, the orthogonal projection $\Pi: \NS(S) \to A(S)^\perp$ induces an isomorphism of groups
$$\MW(S) \to A(S)^{\perp}.$$
\begin{lemma}\label{gplaw} Let $\sigma$ be the class of a section curve, and $\sigma^{*k}$ its $k$-th power with respect to Mordell-Weil group law.  Then the class of $\sigma^{*k}$ in $\NS(S)$ is given by
$$ \left[ \sigma^{*k} \right] = k\sigma - (k-1)z + (z\cdot \sigma+3)k(k-1) f $$
\end{lemma}
\begin{proof}
The class on the generic fiber is computed using the Abel-Jacobi map for elliptic curves.  To determine the coefficient of $f$, use the fact that any section curve has self-intersection $-3$.
\end{proof}
\noindent
For regular surfaces $S$ as above, $A(S)=\langle f,z\rangle$ because the 36 singular fibers are all irreducible nodal cubics (type $I_1$).  If $L$ meets the discriminant curve $\Delta$ at a cusp point, transverse to the cusp direction, the surface $S=\pi^{-1}(L)$ remains regular and $A(S)=\langle f,z\rangle$.  The number of singular fibers drops by one, with two nodal fibers replaced by a single cuspidal fiber (type $II$).\\\\
On the other hand, if $L$ lies in the dual to the discriminant curve $\Delta^*\subset \P^{2*}$, the surface $S=\pi^{-1}(L)$ acquires an $A_n$ singularity (see Appendix A).  It has a minimal resolution $\widetilde{S}$ with $A(\widetilde{S})$ spanned by $f$, $z$, and the exceptional curves.\\\\
To sum up, the two possible sources of jumping Picard rank in our situation are:
\begin{itemize}
\item Resolved singular surfaces.  If $e$ is the class of an exceptional curve, then $\langle f,z,e\rangle$ has intersection matrix
$$\begin{pmatrix}  0 & 1 & 0 \\ 1 & -3 & 0 \\ 0 & 0 & -2 \end{pmatrix}.$$
\item Surfaces with non-trivial Mordell-Weil group.  If $\sigma$ is the class of a section, then $\langle f,z,\sigma\rangle$ has intersection matrix
$$\begin{pmatrix}  0 & 1 & 1 \\ 1 & -3 & z\cdot\sigma \\ 1 & z\cdot\sigma & -2 \end{pmatrix}.$$
\end{itemize}
If $\NS_0(S)=\langle f,z\rangle^\perp \subset \NS(S)$ and $A_0(S) = \NS_0(S)\cap A(S)$, then
$$0\to A_0(S) \to \NS_0(S) \to \MW(S)\to 0$$
is (split) short exact.  This sequence is more germane to future considerations, where we want to ignore the standard classes $z$ and $f$.  The lattice $A_0(S)$ can be either the $A_1$, $A_1\times A_1$, or $A_2$ root lattice.  If we identify the Mordell-Weil group with the sublattice $A(S)^\perp$, then it consists of orthogonal projections of section classes, which have self-intersection $-2(z\cdot \sigma+3)$ (Lemma ~\ref{shift33}).
\begin{remark}  Every vector in the sublattice $\MW(S)$ corresponds to the class of some section curve.  On the other hand, classes in $A_0(S)$ are less geometric:  they are arbitrary $\Z$-linear combinations of $(-2)$-curve classes, whose presence will be accounted for by theta series.\\\\\\\end{remark}

\section{Deformation Theory}
\noindent
The deformation theory of a smooth surface $S$ is controlled by coherent cohomology groups of the tangent sheaf, $T_S$.  In this section, we first consider smooth surfaces $\pi:S\to \P^1$ of type $E(3)$.  Later, we will see what changes when $A_1$ or $A_2$ singularities appear.
\begin{prop}\label{deftheory}
The tangent sheaf $T_S$ has coherent cohomology groups
\begin{align*}
h^0(T_S) &= 0\\
h^1(T_S) &= 30\\
h^2(T_S) &= 0
\end{align*}
which correspond to automorphisms, deformations, and obstructions, respectively.
\end{prop}
\begin{proof}
A global holomorphic 2-form $0\neq \omega\in H^0(K_S)$ is non-degenerate almost everywhere.  If $S$ had a vector field $0\neq v\in H^0(T_S)$, then $0\neq i_v(\omega)\in H^0(\Omega_S)$.  But $q(S)=h^0(\Omega_S)=0$, so $H^0(T_S)=0$.  We compute the higher cohomology groups using Serre duality:
$$H^i(T_S)=H^{2-i}(\Omega_S\otimes K_S)^\vee.$$
The cotangent short exact sequence for $\pi$
$$0\to \pi^* \Omega_{\P^1} \to \Omega_S \to \Omega_{S/\P^1}\to 0$$
yields the following long exact sequence (after tensoring with $K_S$):
\begin{align*}
0 &\to H^0( \pi^* \Omega_{\P^1}\otimes K_S ) \to H^0(  \Omega_S\otimes K_S ) \to H^0(  \Omega_{S/\P^1}\otimes K_S ) \\
  &\to H^1( \pi^* \Omega_{\P^1}\otimes K_S ) \to H^1(  \Omega_S\otimes K_S ) \to H^1(  \Omega_{S/\P^1}\otimes K_S )\\
  &\to H^2( \pi^* \Omega_{\P^1}\otimes K_S ) \to \cancel{H^2( \Omega_S\otimes K_S)} \to H^2( \Omega_{S/\P^1}\otimes K_S)\to 0
\end{align*}
The key observation is that the relative cotangent sheaf $\Omega_{S/\P^1}$ is isomorphic to a subsheaf of the relative dualizing sheaf $\omega_{S/\P^1}$, since the fibers are nodal.
$$\Omega_{S/\P^1} \simeq \omega_{S/\P^1} \otimes \mathcal I \simeq K_S\otimes \pi^* K_{\P^1}^\vee \otimes \mathcal I,$$
where $\mathcal I$ is the ideal sheaf of the 36 nodes in the fibers.  The Hirzebruch-Riemann-Roch formula gives the sheaf Euler characteristics of the columns above:  $3$, $-30$, $-33$, respectively.  Now,
\begin{align*}
h^0( \pi^* \Omega_{\P^1}\otimes K_S ) &= h^0(\P^1,\O(-1))=0\\
h^2( \pi^* \Omega_{\P^1}\otimes K_S ) &= h^0(\pi^*\Omega_{\P^1}^\vee) = h^0(\P^1,\O(2))=3
\end{align*}
imply that $h^1( \pi^* \Omega_{\P^1}\otimes K_S )=0$.  Finally,
$$h^0( \Omega_{S/\P^1}\otimes K_S ) = h^0( \pi^* \O(4)\otimes \mathcal I)=0$$
allows us to deduce the remaining dimensions in the long exact sequence.
\end{proof}
\noindent
The deformations of $S$ are unobstructed, since $H^2(S,T_S)=0$.  By standard lifting results, the first-order deformation space
$$\Def_1(S) = H^1(S,T_S)\simeq \C^{30}$$
admits a versal family of complex surfaces.  If $\beta\in \NS(S)$, is an algebraic class, then 
$$\Def_1(S,\beta)=\ker(\smile\beta : H^1(S,T_S)\to H^2(S,\O_S))\subset \Def_1(S)$$
is the tangent space to the subfamily where $\beta$ remains a $(1,1)$-class.  The dimension of this subfamily depends on the class $\beta$:
\begin{itemize}
\item If $\beta$ is the class of a section, then $\Def_1(S,\beta)\subset \Def_1(S)$ is codim 2.
\item If $\beta$ is the class of an exceptional curve, then $\Def_1(S,\beta)\subset \Def_1(S)$ is codim 1.
\end{itemize}
To prove the first statement, note that the existence of a section implies that $S$ has a Weierstrass model relative to that section.  This construction behaves well in families (see \cite{mirmoduli}), so the dimension of the subfamily can be computed by counting dimensions in the Weierstrass equation, as in the proof of Proposition \ref{generic}.  For the second statement, we introduce the deformation theory of singular surfaces, since exceptional curves arise as resolutions thereof.\\\\
The first-order deformation space of a singular variety is given by
$$\Def_1(S) = \text{Ext}^1(\Omega_S,\O_S).$$
The arguments for vanishing of $H^0(\pi^*\Omega_{\P^1}\otimes K_S )$ and $H^0 (\Omega_{S/\P^1}\otimes K_S)$ carry over the singular case (at least when the number of singular points is small), so we have
$$H^2(T_S)=0.$$
The local-to-global spectral sequence yields
$$0 \to H^1(T_S) \to \text{Ext}^1(\Omega_S,\O_S) \to H^0(\mathcal Ext^1(\Omega_S,\O_S)) \to H^2(T_S)=0. $$
where $\mathcal Ext^1(\Omega_S,\O_S))$ is a skyscraper sheaf supported on the singular locus.  The same spectral sequence also gives $\Ext^2(\Omega_S,\O_S)=0$, so we have a versal family, which fibers over the local deformations of each singular point.\\\\
The local deformation theory of a rational double point has a description due to Brieskorn \cite{brieskorn} in terms of the corresponding ADE root system.  We consider only the $A_n$ case:
$$\dim \mathcal Ext^1(\Omega_{A_n},\O_{A_n}) = n,$$
and the base of the versal family can be taken as $\C^n/\mathfrak S_{n+1}$, where $\mathfrak S_{n+1}$ acts as the Weyl group on the (complexified) $A_n$ root system.  This quotient is always regular (via elementary symmetric polynomials $s_i$).  Following \cite{katzgoren}, we use the simplicial coordinates:
$$\C^n/\mathfrak S_{n+1} \simeq \Spec \left(\C[t_1,\dots,t_{n+1}]\big /\sum_{i=1}^{n+1} t_i \right)^{\mathfrak S_{n+1}} = \Spec \C[s_2,\dots ,s_{n+1}],$$
and the versal family $\mathfrak X$ is cut out of $\C^3\times (\C^n/\mathfrak S_{n+1})$ by the equation
$$uv + w^{n+1}+\sum_{i=2}^{n+1} s_i w^{n-i}.$$
After performing a Galois base change with group $\mathfrak S_{n+1}$, the family admits a {\it simultaneous resolution}, i.e. a proper birational morphism from a smooth family, which resolves the singularities in all fibers.  After base change, the new family $\mathfrak X'$ is cut out by
$$uv + \prod_{i=1}^{n+1} (w+t_i)$$
and has an isolated singularity in the total space.  The fibers have $A_1$ singularities over the generic points of the divisors $(t_i=t_j)$, which are precisely the orthogonal complements of root vectors.  The simultaneous resolution is given by the graph closure of the following rational map from $\mathfrak X' \to (\P^1)^n$:
$$(u,v,w,t_1,\dots, t_{n+1}) \to \left[ u, \prod_{i=1}^k  (w+t_i) \right]_k.$$
There is a bijection between exceptional curves $E$ in the central fiber and vectors $v$ in the root system, and each exceptional curve deforms along the corresponding divisor in $v^\perp\subset \C^n$.\\\\
Since fibered products of analytic spaces can be computed locally in a neighborhood of the fiber singularities, the versal family of a singular $S$ admits a simultaneous resolution after finite base change.  Therefore, exceptional curves persist over a divisor in the base of the deformation.\\\\

\section{Periods}
\noindent
For $S=\pi^{-1}(L)$ a regular elliptic surface as described above, its polarized cohomology lattice
$$\Lambda(S):= \langle f,z\rangle^\perp \subset H^2(S,\Z),$$
which is unimodular and even of signature $(4,28)$, has a polarized pure Hodge structure of weight 2 and type $(2,28,2)$.  By the Lefschetz $(1,1)$ Theorem, 
$$\NS_0(S)\simeq (H^{2,0}(S))^\perp \cap \Lambda(S).$$
There is a period domain which classifies polarized\footnote{From now on, all Hodge structures are polarized.} Hodge structures of this type.  If we fix the abstract lattice $\Lambda$, the space of Hodge structures on it is a homogeneous space for the real Lie group $O(4,28)$:
$$\widetilde{D} \simeq O(4,28)/U(2)\times O(28).$$
It can also be viewed as an open orbit inside a (complex) flag variety, so in particular it has a complex structure.  It contains a sequence of Noether-Lefschetz loci, which are given by
$$\widetilde\NL_{2r} := \bigcup_{\beta\in\Lambda,\,(\beta,\beta) =-2r} \{u\in \widetilde{D}: H^{2,0}(u)\subset \beta^\perp \},$$
each of which is simultaneously a homogeneous space $O(4,27)/U(2)\times O(27)$ and a complex submanifold of $\C$-codimension 2.  They parametrize Hodge structures that potentially come from a surface $S$ with $\NS_0(S)\neq 0$.\\\\
Let $(p,q):\LL\to \P^2\times \P^{2*}$ be the universal line.  Consider the family of elliptic surfaces
$$\S: = X \us{\pi,\P^2,p}\times \LL \os{q}\longrightarrow \P^{2*} $$
whose fiber over $L\in \P^{2*}$ is $\pi^{-1}(L)$.  The family is smooth over $U=\P^{2*}-\Delta^*$.  If $0\in B\subset U$ is an analytic open ball, then a choice of isomorphism $\Lambda(\S_0)\simeq \Lambda$ gives a period map $\tilde{j}:B\to \widetilde{D}$.  From the work of Griffiths, $\tilde{j}$ is holomorphic and horizontal in the sense of \cite{schmid}.  
\begin{lemma} The local period map $\tilde{j}$ is an immersion.\end{lemma}
\begin{proof} By the infinitesimal Torelli theorem of M. Saito \cite{saito}, the differential of the period map from $H^1(T_{\S_0})\to T_{\tilde{j}(0)}\widetilde{D}$ is injective.  Thus, it suffices to check that the Kodaira-Spencer map $KS: T_0\P^{2*} \to H^1(T_{\S_0})$ is injective.  We consider the infinitesimal variation of the zero locus of the form $a_{12}$ restricted to the line $L$, in the moduli space of point configurations on $\P^1$:  $\P^{12}\sslash PGL(2)$.  By generality of $X$, we may assume that the curve $A = Z(a_{12})\subset \P^2$ has no tritangent lines, so there are always $\geq 8$ reduced points in $a_{12}|_L$.  Let $p$ and $q$ be two such points.  Choose coordinates on $\P^2$ so that $p$ is the origin, $L$ is $x$-axis, and $\mathbb T_q A\simeq \P^1$ is the line at infinity.  The pencils $p^\vee$ (resp. $q^\vee$) of lines through $p$ (resp. $q$) represent two independent deformation directions in $T_{L}\P^{2*}$.  Projection to the $x$-axis gives us an isomorphism of nearby lines with $L$, and the slopes of $A$ at the 6 remaining points determine the variation in the configuration space (if the slopes are all equal, then the variation is 0, since we may rescale along $L$, having fixed $p$ and $q$ at $0$ and $\infty$).  In terms of slopes, the tangent space to the variation of the $8$ points in configuration space is isomorphic to $\C^6/\C$, where $\C$ is diagonally embedded.  Now, we set up an incidence correspondence
$$\Omega = \{ (p,q,L,A): p,q\in (L \cap A)_{sm},\,\, p^\vee\text{ and } q^\vee\text{ give linearly dependent variations} \}\subset \mathcal L\times \P^{90}.$$
Via left projection, $\Omega$ surjects onto the universal line $\mathcal L$, which has dimension 4.  The fibers of this surjection are codimension 6 in $\P^{90}=\P H^0(\P^2,\O(12))$, since containing $p$ and $q$ give 2 linear conditions, and the dependence of the slopes give 4 more conditions.  Hence, the right projection maps $\Omega$ onto a proper subset of $\P^{90}$.\end{proof}
\noindent
If $q_U:\S_U\to U$ is the family restricted to $U$, we have a local system $R^2q_{U_*}(\ul{\Z})$, whose fiber over $0\in U$ is the full cohomology lattice $H^2(S_0,\Z)$.  Since $f$ and $z$ are orthogonal to the vanishing cycles everywhere, we have a sub-local system
$$\L_U \subset R^2q_{U_*}(\ul{\Z})$$
whose fiber over $0\in U$ is the polarized cohomology lattice $\Lambda(\S_0)$.  Both of these local systems give variations of Hodge structure over $U$, after tensoring with $\O_U$.  Neither system extends over all of $\P^{2*}$ because the Gauss-Manin connection has non-trivial monodromy $\pi_1(U,0)\to O(\Lambda(\S_0))$.  In order to resolve the ambiguity of the isomorphism $\Lambda(\S_0)\simeq \Lambda$, we quotient the period domain by the arithmetic group $O(\Lambda)$, to obtain
\begin{align*}
D &:= O(\Lambda) \bs \widetilde{D}\\
\NL_{2r} &:= O(\Lambda) \bs \widetilde\NL_{2r}.
\end{align*}
The map $\tilde{j}$ now extends to a well-defined period map $j: U\to D$.  Since $\widetilde{D}$ is not Hermitian symmetric, $D$ is far from being algebraic \cite{grt}; it is one of the non-classical period domains.  Furthermore, the group $O(\Lambda)$ has torsion elements which act with fixed points, so $D$ will have singularities.  In later sections, we will consider the stack quotient instead.
\begin{prop}\label{pdmap} The period map $j:U\to D$ extends to all of $\P^{2*}$.\end{prop}
\begin{proof}  Consider first an analytic neighborhood $B$ of a smooth point of $\Delta^*$.  The intersection $U\cap B$ retracts onto a circle, and the Gauss-Manin connection on this neighborhood has monodromy representation 
$$\pi_1(U\cap B, 0)\simeq \Z\to \Z_2\subset O(\Lambda(\S_0)),$$
sending a generator to the order 2 Picard-Lefschetz twist.  If we base change the family over a $2:1$ cover of $U$, branched along $\Delta^*\cap B$, the monodromy becomes trivial, and in fact we can simultaneously resolve the fibers.  Since we have a smooth family, the period map extends over the entire ball $\widetilde{B}$ upstairs:
$$
\xymatrix{
\widetilde{B}\ar[r] \ar[d] &\widetilde{D}\ar[d]  & \\
B  \ar@{-->}[r] & \Z_2 \bs \widetilde{D} \ar[r]  & D,
}
$$
so we get an induced map on quotients.  When $B$ is a neighborhood of a node in $\Delta^*$, the intersection $U\cap B$ retracts onto a torus, and the Gauss-Manin connection has monodromy representation
$$\pi_1(U\cap B,0) \simeq \Z\times \Z \to \Z_2\times \Z_2 \subset O(\Lambda(\S_0)).$$
The two vanishing cycles are disjoint, so their Picard-Lefschetz twists commute.  When $B$ is a neighborhood of a cusp in $\Delta^*$, the intersection $U\cap B$ retracts onto the trefoil knot complement, and the Gauss-Manin connection has monodromy representation
$$\pi_1(U\cap B,0) \simeq Br_3 \to S_3\subset O(\Lambda(\S_0)). $$
In both situations, one can base change the family over a Galois cover of $U\cap B$, to remove the monodromy, and then simultaneously resolve the fibers.  This allows us to extend the period map on the Galois cover, and then descend the extension to $B$.\end{proof}
\noindent
With Proposition ~\ref{pdmap} in hand, it makes sense to consider $j(\P^{2*})\subset D$.  We are interested in how many times $j(\P^{2*})$ meets the various Noether-Lefschetz loci, a computation which can be done in the cohomology of $D$.  The local simultaneous resolution picture reveals that when $S=\pi^{-1}(L)$ is singular, i.e. when $L\in \Delta^*$, the Hodge structure $j(L)$ is that of the resolution $\widetilde{S}$.  In particular, all such surfaces are Noether-Lefschetz special, because they contain exceptional curves.\\\\\\

\section{Kudla-Millson Modularity}
\noindent
The work of Kudla-Millson produces a modularity statement for Noether-Lefschetz intersection numbers in a general class of locally symmetric spaces $M$ of orthogonal type.  In this section, we summarize\footnote{We match the notation of \cite{km} for the most part, but all instances of positive (resp. negative) definiteness are switched.} the material in \cite{km}, and point out what is necessary to apply it to the period map $j:\P^{2*}\to D$ defined above.\\\\
Let $M$ be the double quotient $\Gamma \bs O(p,q) / K$ of an orthogonal group on the left by a torsion-free arithmetic group (preserving a lattice $\Lambda\subset \R^{p+q}$), and on the right by a maximal compact.  This is automatically a manifold, since any torsion-free discrete group acts freely on the compact cosets.\\\\
We can interpret $\widetilde{M} := O(p,q)/K$ as an open subset of the real Grassmannian $Gr(p,p+q)$ consisting of those $p$-planes $Z\subset \R^{p+q}$ which are positive definite.  For any negative definite line $\ell\subset\R^{p+q}$, set
$$\widetilde{M}_\ell := \{ Z\in \widetilde{M}:  Z\subset \ell^\perp \}$$
which is $\R$-codimension $p$.  The normal bundle to $\widetilde{M}_\ell$ has fibers
$$N_{\widetilde{M}_\ell/\widetilde{M}}(Z) =  \Hom(Z,\ell^\perp).$$
While the image of $\widetilde{M}_\ell$ in $M=\Gamma \bs \widetilde{M}$ may be singular, it can always be resolved if we instead quotient by a finite index normal subgroup of $\Gamma$.  Furthermore, [KM] gives a coherent way of orienting the $\widetilde{M}_\ell$, so that it makes sense to take their classes in the Borel-Moore homology group $H^{BM}_{pq-p}(M)\simeq H^p(M)$.\\\\
Next, the integral lattice $\Lambda\subset \R^{p+q}$ comes into play.  For any negative integer $n$, the action of $\Gamma$ on the lattice vectors of norm $n$ has finitely many orbits (Borel).  Choose orbit representatives $\{y_1,\dots, y_k\}$, and set
$$\widetilde{M}_n := \bigcup_{i=1}^k \widetilde{M}_{\R y_i}$$
The image of $\widetilde{M}_n$ in the arithmetic quotient $M$ is denoted by $M_n$.  Locally, $M_n$ is a union of smooth (real) codimension $p$ cycles meeting pairwise transversely, one for each lattice vector of norm $n$ orthogonal to $Z$.  We quote the following result from \cite{km}:
\begin{theorem}\label{km} For any homology class $\alpha\in H_p(M)$, the series
$$ \alpha\cap\epsilon_p + \sum_{n=1}^\infty e^{\pi i n\tau}\left( \alpha\cap [M_n] \right)$$
is a classical modular form for $\tau\in \mathbb H$ of weight $(p+q)/2$.  The constant term is the integral of the invariant Euler form $\epsilon_p\in H^p(M)$, defined on the symmetric space $\widetilde{M}$.\end{theorem}
\noindent
The proof of Theorem ~\ref{km} utilizes the cohomological theta correspondence, which relates automorphic forms for orthogonal groups and symplectic groups.  In our case, $Sp(1)\simeq SL(2)$ gives a classical modular form, but their result is framed in the higher rank context of Siegel modular forms.\\\\
To apply this statement to our situation, we take $\Lambda \simeq E_8(-1)^3\oplus H^4$, the unique even unimodular lattice of signature $(4,28)$, so we are working with the symmetric space for $O(4,28)$.  The idea is to set $\alpha=j_*[\P^2]\in H_4(D)$, but there are a few technical issues:
\begin{itemize}
\item The isotropy group of the period domain is not maximal compact, so the symmetric space $\widetilde{M}$ is not the same as the local period domain $\widetilde{D}$.
\item The arithmetic group $O(\Lambda)$ has torsion, since the image of the monodromy of the family $\S_U\to U$ does.\\
\end{itemize}
To address the first issue, note that we have a coset fibration
$$ U(2)\to O(4) \to S^2,$$
so the period domain surjects $\widetilde{D}\twoheadrightarrow\widetilde{M}$ with $S^2$ fibers.  This map is not holomorphic, but it is topologically proper.  The Noether-Lefschetz loci $\widetilde{\NL}_{2r}$ are simply the pullbacks of the $\widetilde{M}_{2r}$ via this map.\\\\
The fact that $O(\Lambda)$ has torsion elements means that $D$ has singularities.  Instead of dealing with these directly, we introduce the stack $\D = [O(\Lambda) \bs \widetilde{D} ]$, whose coarse space is $D$.
\begin{lemma}  $O(\Lambda)$ contains a finite index normal subgroup $\Gamma$ which is torsion-free.\end{lemma}
\begin{proof}  Consider the congruence subgroup $\Gamma(3)$, consisting of automorphisms congruent to the identity modulo $3\Lambda$.  It is normal and finite index because the quotient $O(\Lambda)/\Gamma(3)$ lies inside the automorphism group of $\Lambda/3\Lambda$, a finite set.  Suppose that $g\in \Gamma(3)$ is torsion, say of prime order $p$.  Write $1-g = 3a$ for some $a\in \End(\Lambda)$.  Since $g$ has an eigenvalue $\lambda$ of order $p$, $a$ has an eigenvalue $\omega$ with $1-\lambda = 3\omega$.  Now take the norm $N_{\Q(\zeta_p)/\Q}$ of this equation to obtain
$$p = 3^{p-1} \cdot N(\omega)$$
which is never true in $\Z$.  \end{proof}
\begin{remark}
No lattice vector $v$ is conjugate to $-v$ under the action of $\Gamma(3)$, so all the Noether-Lefschetz classes appear with multiplicity two.
\end{remark}
\noindent
The stack $\mathfrak D$ is the quotient of a complex manifold by a finite group.  Indeed, the group $G=O(\Lambda)/\Gamma$ is finite, so
$$\mathfrak D \simeq [ G \bs (\Gamma\bs \widetilde{D}) ]$$
is a smooth analytic stack of Deligne-Mumford type.  Intersection theory on such stacks was constructed in \cite{gil}, \cite{vist}.\\\\
To summarize, the structures described above are related by the diagram
$$\xymatrix{
 & \ov{D}= \Gamma\bs \widetilde{D} \ar[rd]^\mu \ar[ld]_\nu & \\
\D & & M = \Gamma\bs \widetilde{M}
}$$
where $\nu$ is a $G$-cover, and $\mu$ is a sphere fibration.  The modularity statement carries over to intersections on $\D$ as follows.  If $\alpha\in H_4(\D)$ is a homology class, then
\begin{align*}
\alpha \cap \nu_*[\NL_{2r}] &= \nu^*\alpha \cap [\NL_{2r}]\\
 &= \nu^*(\alpha)\cap \mu^*[M_{2r}]\\
 &= \mu_*\nu^*(\alpha)\cap [M_{2r}],
\end{align*}
by repeated applications of the push-pull formula, which still works for smooth stacks of DM type.  Since the locus $\NL_{2r}$ inside $\Gamma\bs \widetilde{D}$ is $O(\Lambda)$-invariant, its pushforward under $\nu$ acquires a multiplicity of $|G|$.  This overall factor can be divided out from the generating series.  After these minor adjustments, we can modify Theorem \ref{km} to our Hodge theoretic situation:
\begin{theorem}\label{km2}  If $\NL_{2r} \subset \D$ denotes the image of $\widetilde{\NL}_{2r}\subset \widetilde{D}$, then
$$\varphi(q) = c_0 + \sum_{r=1}^\infty q^r (\alpha \cap [\NL_{2r}])$$
is a modular form of weight 16 and level $SL(2,\Z)$, for any $\alpha\in H_4(\D)$.\\\\\\ \end{theorem}

\section{Stacks}
\noindent
The period domain $D= O(\Lambda) \bs \widetilde{D}$ is singular, since there is torsion in the arithmetic group $O(\Lambda)$.  Since intersection theory on singular spaces is ill-behaved, we introduced the (smooth) stack quotient
$$\D := [O(\Lambda) \bs \widetilde{D} ].$$
The period map $j: \P^{2*}\to D$ does {\it not} lift to a map from $\P^{2*}\to \D$.  In order to get a lift, we need to define a DM stack $\PP$ over $\P^2$, which is an isomorphism away from $\Delta^*$.  There is a well-defined period map of stacks $j:\PP\to\D$, an immersion whose image is a closed substack of $\D$, with cycle class $\alpha\in H_4(\D)$.  This lifts the global period map of Proposition ~\ref{pdmap}, namely we have a commutative square
$$
\xymatrix{
\PP \ar[r] \ar[d] & \D \ar[d] \\
\P^{2*} \ar[r] & D.
}
$$
In this section, we flesh out this construction, and explain how to compute intersection numbers with Noether-Lefschetz classes on the quotient stack $\D$.  By a result of Artin \cite{artinsimul}, it is always possible to simultaneously resolve a family of surfaces with at worst ADE singularities, after finite base change in the category of algebraic spaces.  We prefer this stacky approach because it is more explicit, and allows us to work locally. \\\\
Recall that morphisms to a quotient stack are characterized by principal bundles.  Namely, if $Y$ is an analytic space, a morphism $Y\to [O(\Lambda)\bs \widetilde{D}]$ is the same as 
$$
\xymatrix{
E \ar[r]\ar[d] & \widetilde{D}\\
Y &
}
$$
where $E\to Y$ is an $O(\Lambda)$-principal bundle, and $E\to \widetilde{D}$ is $O(\Lambda)$-equivariant.  If $\L$ is a local system on $Y$ with fiber $\Lambda$, equipped with a variation of Hodge structure, we get a morphism to $[O(\Lambda)\bs \widetilde{D}]$ as follows.  Take $E$ to be the sheaf of isomorphisms from $\L$ to $\ul{\Lambda}$, the constant sheaf, and define the equivariant map by sending $(y,\phi)$ to the Hodge structure on $\Lambda$ induced by $\phi:\L(y)\os{\sim}\to \Lambda$.  This characterization can be extended to the case where $Y$ is itself a DM stack, using descent theory.\\\\
The construction of $\PP$ is designed to introduce stacky structure on $\P^{2*}$ in accordance with the monodromy described in the proof of Proposition \ref{pdmap}.  The advantage is that the local system $\L_U \subset R^2q_{U*}(\ul{\Z})$ over $U\subset \P^{2*}$ extends to all of $\PP$, and admits a variation of Hodge structure, so we get a period map $\PP\to \D$.\\\\
{\bf Construction of $\PP$}.  First, take the double over $T\to \P^{2*}$ branched along $\Delta^*$.  If $F_{\Delta^*}\in H^0(\P^{2*},\O(612))$ is the form cutting out $\Delta^*$, we can set
$$T = Z(u^2 = F_{\Delta^*})\subset \text{Tot}(\O_{\P^{2*}}(612)),$$
where $u$ is the fiber coordinate of the line bundle $\O_{\P^{2*}}(311)$.  This new surface will have isolated singularities over the singularities of $\Delta^*$.  Over each node of $\Delta^*$, $T$ has an $A_1$ singularity, and over each cusp of $\Delta^*$, $T$ has an $A_2$ singularity.  This follows immediately from the local equations:
\begin{align*}
A_1 &:=Z(u^2 = x^2+y^2)\subset \C^3\\
A_2 &:=Z(u^2 = x^3+y^2)\subset \C^3.
\end{align*}
On the other hand, recall that the $A_n$ singularities are also cyclic quotients of $\C^2$:
\begin{align*}
A_1 &\simeq \C^2/\Z_2\\
A_2 &\simeq \C^2/\Z_3.
\end{align*}
Following \cite{gil}, we can construct a stack $\mathfrak T$ whose coarse space is $T$, with the coarse map given \'{e}tale locally by the usual
\begin{align*}
[\C^2/\Z_2]&\to \C^2/\Z_2\\
[\C^2/\Z_3]&\to \C^2/\Z_3.
\end{align*}
This is done by first covering $T$ by \'{e}tale neighborhoods $U_\alpha$ which are isomorphic to cyclic quotients $V_\alpha/G_\alpha$.  Let $V$ be the disjoint union of the $V_\alpha$'s, and form the groupoid scheme
$$ V\us{T}\times V \rightrightarrows V.$$
The structure maps are both \'{e}tale, so the associated stack $\T$ is DM, and the fact that its coarse space is $T$ can be checked \'{e}tale locally.  The involution $T\to T$ coming from the double cover extends to $\T$ by Lemma \ref{romagny} below, as long as we choose the $U_\alpha$ to be preserved by the involution.  Finally, we define $\PP := [\T/\Z_2]$.
\begin{lemma}\label{romagny}
If $X$ admits a $G$-action, and $H \trianglelefteq G$ is a normal subgroup, then the stack $[X/H]$ has an action of $G/H$ such that 
$$[[X/H]/(G/H)]\simeq [X/G]$$
\end{lemma}
\begin{proof}
See Remark 3.4 of \cite{romagny}.  Note that we only need the cases $\Z_3 \trianglelefteq S_3$ and $\Z_2 \trianglelefteq \Z_2\times\Z_2$.
\end{proof}
\noindent
The local system $\L$ on $\P^{2*}-\Delta^*$ extends over all of $\PP$ by analytic local picture, where $\PP$ is a quotient stack of a complex ball $B\subset \C^2$.  A local system on the quotient $[B/G]$ is the same as a local system on $B$, which must be trivial, together with an action of $G$ on the fiber (the local monodromy action).\\\\
To compute the intersection numbers, we spread out the calculation over $\PP$.  Given a local system $\L$ on $Y$ underlying a variation of Hodge structure, we can associate to it a fiber bundle $\widetilde{D}(\L)\to Y$ with fiber $\widetilde{D}$, and a canonical section $s$ of this bundle which picks out the Hodge structure on $\L(y)$.  To construct this family, note that $O(\Lambda)$ acts on $\widetilde{D}$, so we just need to adapt the principal bundle to fiber bundle construction to the context of stacks:
$$\widetilde{D}(\L) := \widetilde{D} \times_{O(\Lambda)} E = [(\widetilde{D} \times E)/O(\Lambda)]\to \D$$
which can also be done by descent.  As for the section, we need a morphism from $\PP$ to the above stack quotient, so take the principal bundle $E$ itself, and the graph of the map $E\to \widetilde{D}$.  Now the diagram is
$$
\xymatrix{
\widetilde{D}(\L) \ar[d] \ar[rd] & \\
\PP \ar@/^/[u]^s \ar[r] & \D
}
$$
The Noether-Lefschetz loci can also be spread out to $\widetilde{\NL}_{2r}(\L)\subset \widetilde{D}(\L)$ over $\PP$, but note that there are now infinitely many components in every fiber of $\widetilde{D}(\L)$.  To remedy this, we replace $\widetilde{D}$ with the smooth arithmetic quotient $\ov{D}=\Gamma \bs \widetilde{D}$ from Section 4 to form the fiber bundle:
$$\ov{D}(\L) := \ov{D} \times_{O(\Lambda)} E = [(\ov{D} \times E)/O(\Lambda)]\to \D$$
Now there are finitely many Noether-Lefschetz components in each fiber, and they are permuted by the monodromy action, which factors through $G = O(\Lambda)/\Gamma$.  By Lemma 2.1 of \cite{km}, each component of $\ov{\NL}_{2r}\subset \ov{D}$ is smooth (at least after normalization, to remove the normal crossings).  We have a local embedding
 $$\ov{\NL}_{2r}(\L)^{norm} \to \ov{D}(\L)$$
 since $\ov{\NL}_{2r}$ has only normal crossing singularities.\\\\
The desired cohomological intersection numbers are 
$$\int_{\ov{D}(\L)} [s(\PP)] \cap \ov{\NL}_{2r}(\L), $$
where $s$ is the section $\PP \to \ov{D}(\L)$.\\\\
Using Vistoli's formalism \cite{vist}, the Gysin map associated to the local regular embedding $f:\ov{\NL}_{2r}(\L)^{norm} \to \ov{D}(\L)$ is given in terms of the Cartesian square
$$\xymatrix{
W \ar[r] \ar[d]_g & \PP \ar[d]^s \\
\ov{\NL}_{2r}(\L)^{norm} \ar[r]_f & \ov{D}(\L)
}$$
Let $N= g^* N_f$, containing the normal cone $C_{W/\PP}$.  The intersection number is given by
$$f^![\PP] := 0^!_N [C_{W/\PP}].$$\\\\

\section{Gromov-Witten Theory}
\noindent
The moduli space of stable maps $\ov{M}_0(X,\ell+nf)$ can be partitioned into 
$$\ov{M}_0(X,\ell+nf) = W_{n,0} \sqcup W_{n,1}\sqcup \dots \sqcup W_{n,n},$$
where $W_{n,i}$ consists of stable maps $u:C\to X$ whose image consists of a smooth horizontal component of class $\ell+if\in H_2(X)$, glued with $(n-i)$ vertical components, which must be singular fibers, since the geometric genus is 0).
\begin{prop}
Each locus $W_{n,i}$ is both closed and open inside $\ov{M}_0$.
\end{prop}
\begin{proof}
This follows from the fact that arithmetic genus of the image curve is preserved under degeneration.  When $u\in W_{n,i}$, the self-intersection of $u(C)$ inside the elliptic surface $\pi^{-1}(\pi(C))$ is $-3+n-i$.\end{proof}
\noindent
\begin{lemma}\label{shift3}  The inclusion map $\iota:S \to X$ induces a homomorphism $\iota_*:H_2(S)\to H_2(X)$.  If $\sigma$ is the class of a section, then
$$\iota_*(\sigma)= \ell + (\sigma\cdot z+3)f.$$
\end{lemma}
\begin{proof}
This is a direct calculation using the fact that the zero section $z(\P^2)\subset X$ has normal bundle $\O(-3)$.
\end{proof}
\begin{cor}\label{empty}
$W_{n,1}$ and $W_{n,2}$ are empty, for all $n$.
\end{cor}
\noindent
The Gromov-Witten invariant breaks up into a sum:
$$N^X_{0,\ell+nf} = \sum_{i=0}^n \int_{[W_{n,i}]^{vir}} 1.$$
By a dimension count, one expects $W_{n,n}$ to be a finite set when $n>0$:
$$h(n):=\#\{ L\subset \P^2: \,\, \exists \sigma\in\MW(\pi^{-1}(L)), \iota_*(\sigma)=\ell+nf  \}<\infty.$$
This is the case when $X$ is sufficiently general (see Proposition \ref{generic}).  The numbers $h(n)$ control ``most'' of the Gromov-Witten invariant.  For $0<i<n$, $W_{n,i}$ fibers over $W_{i,i}$ (by dropping the vertical components), and the contribution of each fiber is given by the formula of Bryan-Leung:
$$\int_{[W_{n,i}]^{vir}} 1 = h(i) [\eta(q)^{-36}]_{n-i}.$$
The final piece $W_{n,0}$ is harder to compute, since the horizontal component may vary over lines in the zero section $z(\P^2)\subset X$.  We will address this in a future paper.
\begin{prop}\label{generic}
If the forms $a_{12}$ and $b_{18}$ in the definition of $X$ are very general, then $W_{n,n}$ is a finite discrete set.\end{prop}
\begin{proof}
To get the desired transversality, we introduce the (algebraic) moduli space $\W_3$ of Weierstrass fibrations over $\P^1$ constructed in \cite{mirmoduli}.  Using Mumford's geometric invariant theory (GIT), one can form a geometric quotient of 
$$V=H^0(\P^1,\O(12))\oplus H^0(\P^1,\O(18))$$
by the natural action of $\C^*\times SL_2(\C)$, after removing the unstable locus.  A point $(\alpha,\beta)\in V^*$ is GIT-stable if and only if $\alpha$ (resp. $\beta$) never vanishes with multiplicity $\geq 6$ (resp. $\geq 9$).  If the forms $a_{12}$ and $b_{18}$ are general, then the associated plane curves do not have high order tangents, so the family $q:\S\to \P^{2*}$ induces a map to moduli 
$$w_{a,b}:\P^{2*}\to \W_3.$$
The moduli space $\W_3$ has dimension 28 and contains two notable irreducible divisors, $D_c$ and $D_r$.  The generic member of $D_c$ is regular with a cuspidal fiber, and the generic member of $D_r$ is singular, with minimal resolution of Kodaira type $I_2$.  If the curves defined by $a_{12}$ and $b_{18}$ meet transversely, then
\begin{align*}
w_{a,b}^{-1}(D_c)&=\bigcup_{i=1}^{216} c_i^*\\
w_{a,b}^{-1}(D_r)&=\Delta^*
\end{align*}
where $c_i$ are the cusps of $\Delta$.  In \cite{cox}, Cox proves that the Mordell-Weil loci are all reduced of codimension 2 in $\W_3-D_r$.  By Lemma \ref{starr} below, $D_r$ must contain surfaces with trivial Mordell-Weil group, so in fact the Mordell-Weil loci are codimension 2 everywhere.  By a Bertini argument detailed in Lemma \ref{bertini}, very general choices of $(a_{12},b_{18})$ will produce a moduli map $w_{a,b}$ transverse to each of these loci.  However, the intersection of $w_{a,b}(\P^{2*})$ with the Mordell-Weil locus could occur within $D_r$, where the surface in question is singular.
\end{proof}
\begin{lemma} \label{starr}
There exist surfaces in $D_r$ with trivial Mordell-Weil group.
\end{lemma}
\begin{proof}
Fix an extremal rational elliptic surface $R\to\P^1$, with reducible fibers and trivial Mordell-Weil group (these were classified by Miranda).  For each degree 3 cover $f:\P^1\to \P^1$, we obtain an $E(3)$ elliptic surface $S_f\in D_r$ via base change.  
$$\xymatrix{
S_f \ar[r]^{f'}\ar[d] & R \ar[d]\\
\P^1 \ar[r]^f & \P^1
}$$
If each $S_f$ has an extra section, then composing with $f'$, we obtain a family of multisections in $R$:
$$\xymatrix{
\mathcal C\ar[r]\ar[d] & R \\
\ov{M}_0(\P^1,3)  &
}$$
which degenerates to a comb curve at the boundary.  This implies that $R$ has an extra section, contradiction.
\end{proof}
\begin{defn}
Let $Z$, $W$, and $S$ be smooth varieties.  A family 
$$f:Z \times S\to W$$
of immersions is deemed {\it freely movable} if for any $x_0\in Z_0$ and $v\in T_{x_0}W$, there exists a 1-parameter subfamily $T\to S$ and smooth section $x(t)\in Z_t$ such that $x(0)=x_0$ and $\dot{x}(0) = v$.
\end{defn}
\begin{lemma}\label{bertini}
Let $Z\times S\to W$ be a family of immersions which is freely movable.  Then for any fixed smooth subvariety, $\Pi\subset W$, there exists $s\in S$ such that $Z_s$ meets $\Pi$ transversely.
\end{lemma}
\begin{proof}
We argue using local holomorphic coordinates.  Assume that $W=\C^n$ and $\Pi = \C^\ell\times\{0\}$, with standard coordinates $w_1,w_2,\dots,w_n$.  Assume that $f_0:D^r\to W$ sends $f_0(0)=0\in W$, and is not transverse to $\Pi$ at $0\in D^r$.  In terms of coordinates $u_i$ on $D^r$, we know that the vectors
$$\partial_{w_1}, \dots,\partial_{w_\ell}, df_0(\partial_{u_1}), \dots, df_0(\partial_{u_r})$$
do not span $T_0 W$.  Let $v$ be a vector outside their span, and use the hypothesis of free movability to find a subfamily $f_t:D^r\to W$ and $x(t)\in D^r$ such that $\dot{x}(0)=v$.  Suppose that there exist sequences of points $y_i\to 0\in D^r$ and $t_i\to 0\in T$ such that $f_{t_i}(y_i)$ meets $\Pi$ non-transversely.  Then we have
$$f_{t_i}(y_i)- f_0(0) = \left( f_{t_i}(y_i)- f_{t_i}(x(t_i)) \right) + \left( f_{t_i}(x(t_i)) - f_0(0)   \right).$$
To first order, the LHS lies in the span of $\partial_{w_1}, \dots,\partial_{w_\ell}$, and the second term on the RHS lies in the span of $v$.  The first term on the RHS lies in a small deformation of the span of $df_0(\partial_{u_1}), \dots, df_0(\partial_{u_r})$.  This contradicts our choice of $v$, so a sufficiently small deformation $f_t:D^r\to W$ meets $\Pi$ transversely.  \end{proof}
\noindent
To apply this to our situation, take the family of subvarieties $w_{a,b}(\P^{2*})\subset \W_3$, varying $a_{12}\in H^0(\P^2,\O(12))$ and $b_{18}\in H^0(\P^2,\O(18))$.  This family is freely movable because the restriction of forms from $\P^2$ to any line $L\subset \P^2$ is surjective.\\\\
In order to use the technique of Bryan-Leung to compute the contributions of singular surfaces with non-trivial sections, we need to study the deformation theory of elliptic surfaces.  For this, we use the analytic Tate-Shafarevich group of an elliptic surface, which classifies a type of non-algebraic deformations destroying sections.  See the treatment in \cite{morgfried} for more details.\\\\
If $\pi:E\to B$ be a complex elliptic fibration, let $\E$ denote the sheaf of local holomorphic sections of $\pi$, which is a sheaf of abelian group using the fiberwise group law.  Now,
$$\sha^{an}(E):= H^1(B,\E)$$
can be identified with the set of (isomorphism classes of) genus one fibrations $\pi':E'\to B$ equipped with an isomorphism $\Jac(E')\to E$ over $B$.  Here, the relative Jacobian fibration $\Jac(E')\to B$ is the (unique) elliptic fibration locally isomorphic to $\pi':E'\to B$.\\\\
We may understand the sheaf $\E$ in terms of the (split) short exact sequence
$$0\to \E \to (R^1\pi_* \O_E^*)/\A \to \ul{\Z}\to 0,$$
where $\A\subset R^1\pi_* \O_E^*$ is the subsheaf of divisor classes which do not dominate the base.  As long as $\A$ is supported on an affine set, which will be true in our situation, it has no higher cohomology, so we have
$$0\to H^1(\E)\to H^1(R^1\pi_* \O_E^* )\to H^1(B,\Z)\to 0.$$
From the exponential sequence for $E$, and the Leray spectral sequence for $\pi$, we have an exact sequence
$$ H^2(E,\Z)\to H^2(E,\O_E)\os\delta\to  H^1(B,\E) \to H^3(E,\Z) \to H^1(B,\Z).$$
The image\footnote{In the cases relevant to us, $H^3(E,\Z)=0$, so the connecting morphism $\delta$ is surjective.  } of $\delta$ is the connected component of $0\in H^1(B,\E)$.  There is a universal family of genus 1 fibrations over the vector space $H^2(E,\O_E)$, with origin $\pi:E\to B$ (the only member with a section).  We refer to any 1-parameter subfamily as a $\sha$-deformation of $E$.\\\\
We apply this formalism to the case where $B$ is a small analytic neighborhood of a line $L\subset \P^2$ tangent to the discriminant curve $\Delta$, with non-trivial section $\sigma\subset \pi^{-1}(L)$.  Set
$$U = \{ L':  L'\subset B \}\subset \P^{2*}$$
Even though $\pi^{-1}(L)$ is singular, $E=\pi^{-1}(B)$ is a smooth non-compact threefold.  The natural inclusion
$$\ov{M}_0(E,\ell+nf) \to \ov{M}_0(X,\ell+nf)$$
restricts to an isomorphism on components involving $\sigma$, since $L$ is the unique line in $U$ with a non-trivial section.  Consider a small $\sha$-deformation of the local elliptic threefold $E\to B$.  
\begin{lemma}\label{restriction}
For each line $L'\in U$, the restriction from $H^2(E,\O_E)\to H^2(\pi^{-1}(L'), \O_{\pi^{-1}(L')})$ is surjective.
\end{lemma}
\begin{proof}
Consider the normal bundle sequence for the zero section $\Sigma\subset E$
$$0\to \O_E \to \O_E(\Sigma) \to N_{\Sigma/E} \to 0$$
with associated long exact sequence
$$H^1(E,\O_E(\Sigma)) \to H^1\left(\Sigma,N_{\Sigma/ E} \right) \to H^2(E,\O_E) \to H^2(\Sigma,\O_E(\Sigma)).$$
The cohomology groups on the ends vanish, since the Leray spectral sequence for $\pi$ gives
$$H^1(E,\O(\Sigma)) \simeq H^1(B,\pi_*\O(\Sigma)) = H^1(B,\O)=0$$
$$H^2(E,\O(\Sigma)) \simeq H^2(L',\pi_*\O(\Sigma)) = H^1(B,\O)=0.$$
The same argument works for the fibration $\pi^{-1}(L') \to L'$ with zero section $Z$, so we have a commutative diagram with horizontal isomorphisms:
$$\xymatrix{
H^1(\Sigma, N_{\Sigma/E})\ar[r]^\sim\ar[d] &  H^2(E,\O_E) \ar[d]\\
H^1\left (Z,N_{Z/\pi^{-1}(L')}\right ) \ar[r]^\sim &  H^2\left (\pi^{-1}(L'),\O_{\pi^{-1}(L')}\right ).
}
$$
The left hand restriction map is slightly easier to understand.  Surjectivity is equivalent to the vanishing of $H^2(\Sigma, N_{\Sigma/E}\otimes \mathcal I_Z)$, which holds because $\Sigma$ admits a covering by 2 contractible open sets.
\end{proof}
\noindent
In particular, the restriction map
$$H^2(E,\O_E)\twoheadrightarrow H^2\left(\pi^{-1}(L),\O_{\pi^{-1}(L)}\right)$$ 
is surjective, so we can choose a splitting $V\subset H^2(E,\O_E)$.  Since all the surfaces $\pi^{-1}(L')$ have a section, the moduli map
$$V \times U \to \Def_1(\pi^{-1}(L))$$
is injective.  Recall that $U\to \Def_1(\pi^{-1}(L))$ meets the Mordell-Weil locus transversely at $L$, so a small deformation 
$$\{v\}\times U\to \Def_1(\pi^{-1}(L))$$ 
will also meet the Mordell-Weil locus transversely, at a nearby line $L'\neq L$.  We will arrange that $L'\notin \Delta^*$ using Lemma \ref{restriction}.\\\\
Restricting the elements of $V\subset H^2(E,\O_E)$ to each surface in the family $q:\S|_U\to U$, we obtain a 2-dimensional subspace of sections of the rank 2 vector bundle $R^2q_*(\O_\S)$, which is surjective on each fiber.  Thus, the graphs of these sections are freely movable, so there exists a section whose zero locus misses $\Delta^*$.  In other words, there exists a $\sha$-deformation $\pi':E'\to B$ such that
$$\{ L'\in U:  \MW(\pi'^{-1}(L'))\neq 0 \} \cap \Delta^*=\emptyset.$$
By deformation invariance, the components of $\ov{M}_0(X,\ell+nf)$ which involve $\iota(\sigma)$ can be computed on the deformed local threefold, where the section lies on a smooth elliptic surfaces, and we can invoke the formula of Bryan-Leung \cite{bryanleung}:
\begin{theorem}
Let $S$ be a smooth elliptic surface with topological Euler characteristic $12k$ and section $z$.  Then the local Calabi-Yau threefold $Tot(K_S)\to S$ has genus 0 Gromov-Witten invariants given by the generating series:
$$ \sum_n N^{Tot(K_S)}_{0,z+nf}q^n = \prod_{m=1}^\infty (1-q^m)^{-12k} =: \eta(q)^{-12k} . $$\\\\
\end{theorem}

\section{Hodge Bundle}
\noindent
The family of surfaces $q:\S\to \P^{2*}$ has a corresponding Hodge bundle $\mathcal H=q_* (\omega_{\S/\P^{2*}})$ whose fiber at $L\in \P^{2*}$ is the space of holomorphic 2-forms on $S=\pi^{-1}(L)$.  This makes sense even when $S$ is ADE-singular; in fact, for any resolution $\eps:\widetilde{S}\to S$,
$$K_{\widetilde{S}} = \eps^*K_S.$$
This implies that the Hodge bundle for the simultaneously resolved family over $\PP$ is the pull back of $\H$.  The Chern classes of $\mathcal H$ are important to us for two reasons.
\begin{lemma}  The constant term of the Kudla-Millson modular form, given in terms of the Euler form $\eps_4$ on the symmetric space, is given by
$$\int_{\PP} j^* \nu_* \mu^*\eps_4=  \int_{\P^{2*}}c_2(\H).$$
\end{lemma}
\begin{proof}
By Chern-Weil theory, the form $\eps_4$ represents the Euler class of the real vector bundle $V\to M$ defined as follows.  Viewed as the Grassmannian of positive definite 4-planes inside $\R^{4+28}$,  the symmetric space $\widetilde{M}$ has a tautological subbundle $\widetilde{V}\subset M\times \R^{4+28}$.  The action of $\Gamma$ on $\widetilde{M}$ lifts naturally to $\widetilde{V}$, so taking after taking quotients, we are left with the tautological bundle $V\to M$.  After pulling back to the period domain, $\ov{D}$, the complexification splits:
$$\mu^* V\otimes \C \simeq V^{0,2} \oplus V^{2,0}$$
with each summand isomorphic to $\mu^* V$ as a real vector bundle.  Hence,
$$c_2(V^{2,0}) = \mu^*e(V)$$
Since $V^{2,0}$ is the Hodge bundle on the period domain, the result follows from push-pull.
\end{proof}
\begin{lemma}  The normal bundle to the Noether-Lefschetz locus $\NL_2\subset \D$ is isomorphic to the dual Hodge bundle, when pulled back to $\PP$.  \end{lemma}
\begin{proof}
This follows from the description of the normal bundle in Section 4, along with the fact that local systems have trivial Chern classes.
\end{proof}
\noindent
To compute the Chern classes of $\H$, we factor $q:\S\to \P^{2*}$ as
$$\S\os{f}\to \LL \os{g}\to \P^{2*}$$
where $\LL\subset \P^2\times \P^{2*}$ is the universal line, cut out by a form of bidegree (1,1).  First, we compute
\begin{align*}
f_*(\omega_{\S/\P^{2*}}) &= f_*(\omega_{\S/\LL} \otimes f^*\omega_{\LL/\P^{2*}} )\\
 &= f_*(\omega_{\S/\LL}) \otimes \omega_{\LL/\P^{2*}} \\
 &= f_*(\omega_{\S/\LL})\otimes (\omega_{\LL}\otimes g^*\omega_{\P^{2*}}^\vee) \\
 &= \O_\LL(3,0)\otimes (\O_\LL(-2,-2) \otimes \O_\LL(0,3)) \\
 &= \O_\LL(1,1)
\end{align*}
Next, we use $q_* = g_*f_*$ to obtain $q_* (\omega_{\S/\P^{2*}}) = g_*\O_\LL(1,1)$.  We have the short sequence
$$0\to \O_{\P^2\times \P^{2*}}\to \O_{\P^2\times \P^{2*}}(1,1)\to \O_\LL(1,1)\to 0$$
so in K-theory $[\O_\LL(1,1)] = [\O_{\P^2\times \P^{2*}}(1,1)] - [\O_{\P^2\times \P^{2*}}]$.  Since these sheaves have no higher cohomology, the Chern classes depend only on the pushfoward of $\O_{\P^2\times \P^{2*}}(1,1)$ to the second factor:
\begin{align*}
ch(g_*[\O_\LL(1,1)])  &= ch\left(\O_{\P^{2*}}(1)^{\oplus 3}\right) - 1.
\end{align*}
Therefore, the Hodge bundle has Chern classes $c_1(\H)=3h$ and $c_2(\H)=3h^2$.\\\\\\

\section{Enumerative Features}
\noindent
In this section, we discuss how to obtain the section counts $h(n)$ from the modularity statement.  Since there are very few modular forms of weight 16:
$$\dim M_{16}(\Gamma(1))=2,$$
it is possible to pin down the $q$-series $\varphi(q)$ of the modular form from Theorem ~\ref{km2}, with $\alpha = j_*[\PP]$.  First, recall the Noether-Lefschetz setup, which is concerned with the orthogonal projection of $\sigma$ to the primitive sublattice $\Lambda_0$:
\begin{lemma}\label{shift33}
The orthogonal projection of a non-trivial section class $\sigma$ to the primitive lattice $\NS_0(S)$ has self-intersection of $-2(\sigma\cdot z+3)$.
\end{lemma}
\begin{proof}
This is a direct calculation using the polarization sublattice $\langle f,z \rangle$ introduced in Section 2.
\end{proof}
\noindent
This matches up well with Lemma \ref{shift3}, and it implies that non-trivial section counts only begin to contribute to $\varphi(q)$ at order $q^3$.  The surfaces with $A_n$ singularities contribute to $\varphi(q)$ at orders dictated by the corresponding $A_n$ lattice norm values.  At first pass, one can argue set theoretically.
\begin{itemize}
\item All the fibers of $\S\to\P^{2*}$ with an $A_1$ singularity, which occur over $\Delta^*$, have Hodge structures with a $(-2)$-curve inside the algebraic part.  By taking multiples of this curve class, we find $j(\Delta^*)\subset\NL_2, \NL_8,\NL_{18},\dots$, so these intersections contribute to $q^1,q^4,q^9,\dots$.  
\item All the fibers of $\S\to\P^{2*}$ with an $A_2$ singularity, which occur over cusps of $\Delta^*$, have Hodge structures with a copy of the $A_2$ lattice inside the algebraic part.  By taking different lattice vectors, we find $j(\Delta^*_c)\subset \NL_2, \NL_6,\NL_8$, so these intersections contribute to $q^1, q^3, q^4,\dots$.
\item All the fibers of $\S\to\P^{2*}$ with two $A_1$ singularities, which occur over nodes of $\Delta^*$, have Hodge structures with a copy of the $A_1\oplus A_1$ lattice inside the algebraic part.  By taking different lattice vectors, we find $j(\Delta^*_n)\subset \NL_2, \NL_4,\NL_8$, so these intersections contribute to $q^1, q^2, q^4,\dots$.
\end{itemize}
See Appendix A for more details on the singularity types and Pl\"{u}cker formulas.\\\\
The $q^2$ term of $\varphi(q)$ is the intersection $j(\PP)\cap \NL_4$, which occurs only along $\Delta^*_n$, corresponding to those lines which are bitangent to the discriminant $\Delta\subset \P^2$.  To compute the stacky intersection product, we work locally in a neighborhood of the node, where $\PP$ is isomorphic to $[\C^2/\Z_2\times \Z_2]$.  After the (\'{e}tale) base change to $\C^2$, the local system $\L$ trivializes, since the family can be simultaneously resolved.  It meets each branch of the Noether-Lefschetz locus transversely.  There are two branches of $\ov{\NL}_4(\L)$ here, corresponding to $e+e'$ and $e-e'$, but each branch has multiplicity two because their negatives are not conjugate under the torsion-free subgroup $\Gamma(3)$.  Counting each isolated intersection in this way, we obtain 
$$[\varphi]_2 = \frac{4}{4} \cdot 184032.$$
As discussed above, the constant term of $\varphi(q)$ is the Euler class of the Hodge bundle:
$$[\varphi]_0 = 3.$$
These two coefficients are enough to pin down the modular form:
$$\varphi(q) = \frac{31 E_4(q)^4 + 113 E(q)^4E_6(q)^2}{48} =  3 - 1188 q + 184032 q^2 +\dots .$$
Now, we can write down an expression for the section curve counts $h(n)$, in terms of $\varphi(q)$ and correction terms coming from the singular surfaces.
\begin{theorem}
Ignoring constant terms, we have the following equality of $q$-series:
$$\varphi(q)=-\frac{1188}{2} \Theta_1(q) + \frac{184032}{4} (\Theta_1(q) - 1)^2 + \frac{1944}{6} (\Theta_2(q) - 3\Theta_1(q)) + \sum_{n\geq 3} h(n) q^n,$$
where $\Theta_i(q)$ denotes the theta series for the root lattice $A_i$.  
\end{theorem}
\begin{proof}
First, consider the $q^1$ term, which is an excess intersection contribution:
$$[\varphi]_1 = -1188 = j(\PP)\cap \NL_2. $$
In a neighborhood of each cusp of $\Delta^*$, the local monodromy of $\PP$ acts transitively on the lattice vectors which specialize to exceptional $(-2)$-curves in the simultaneous resolution ($e_1$, $e_2$, $e_1+e_2$, and their negatives).  This implies that there is only one irreducible component of $\ov{\NL}_2(\L)$ which intersects $s(\PP)$.  By taking integer multiples of these lattice vectors, of self-intersection $-2k^2$, we find an identical component $\ov{\NL}_{2k^2}(\L)^\circ\subset \ov{\NL}_{2k^2}(\L)$, which also intersects $\PP$ along $\Delta^*$, with the same excess contribution (this is the first term).  But there are also components of $\ov{\NL}_{2k^2}(\L)$ with {\it isolated} intersections with $s(\PP)$ only at nodes and cusps.
\begin{itemize}
\item At a given point $[L]\in \Delta^*_n$, these components correspond to lattice vectors $ae+be'$, where $a\neq 0$ and $b\neq 0$, and $-2a^2-2b^2=-2k^2$.  The number of such lattice vectors is given by
$$\left[\left( \Theta_1(q) - 1 \right)^2\right]_{2k^2}$$
\item At a given point $[L]\in \Delta^*_c$, these components correspond to lattice vectors $ae_1+be_2$, where $a\neq 0$, $b\neq 0$, $a\neq b$, and $-2a^2-2b^2+2ab=-2k^2$.  The number of such lattice vectors is given by
$$\left[\Theta_2(q) - 3\Theta_1(q)\right]_{2k^2}$$
\end{itemize}
The factors of $\frac{1}{4}$ and $\frac{1}{6}$ come from the \'{e}tale base change required to compute the stacky intersection as a transverse intersection.\end{proof}
\vspace{10mm}

\section{Appendix A}
\noindent
Here, we study the singularities that occur in our geometries.\\\\
The morphism $\pi:X\to \P^2$ is an elliptic fibration with an irreducible discriminant curve $\Delta\subset \P^2$ of degree 36, with 216 cusps and no other singularities.  To see this, note that the forms $a_{12}$ and $b_{18}$ give curves in $\P^2$ which meet transversely at $216$ points.  From the form of the discriminant equation
$$\Delta = 4a^3+27b^2,$$
each of these points is a cusp of $\Delta$.  From the weighted hypersurface description, $X$ has topological Euler characteristic $-540$, so there are no further singularities in $\Delta$.
\begin{itemize}
\item  Away from $\Delta$, the fibers of $\pi$ are smooth curves of genus 1 (type $I_0$).
\item When $p\in \Delta$ is a general point, $\pi^{-1}(p)$ is an irreducible nodal cubic (type $I_1$).  An analytic neighborhood of the singularity is a versal family for the node, crossed with a disk:
$$Z(y^2+x^2+t)\to \C^2_{s,t}$$
\item When $p\in \Delta$ is a cusp, $\pi^{-1}(p)$ is a cuspidal cubic (type $II$).  An analytic neighborhood of the singularity is a versal family for the cusp:
$$Z(y^2+x^3+sx+t)\to \C^2_{s,t}.$$
This versal family remains nodal over the discriminant curve given by $4s^3+27t^2$.\\
\end{itemize}
The morphism $q:\S\to \P^{2*}$ is the family of elliptic surfaces obtained by restricting $\pi$ to various lines $L\subset\P^2$.
\begin{itemize}
\item When $L$ meets $\Delta$ transversely at a general point, the surface $\pi^{-1}(L)$ is smooth, given locally by $y^2+x^2+t$, where $t$ is the parameter along $L$.
\item When $L$ meets $\Delta$ with order 2 (tangent) at a general point, the surface $\pi^{-1}(L)$ has an $A_1$ singularity, given locally by $y^2+x^2+t^2$.
\item When $L$ meets $\Delta$ with order 3 (at a flex), the surface $\pi^{-1}(L)$ has an $A_2$ singularity given locally by $y^2+x^2+t^3$.
\item When $L$ meets $\Delta$ at a cusp in the non-special direction ($s=0$), the surface $\pi^{-1}(L)$ is smooth, given locally by $y^2+x^3+t$.
\item When $L$ meets $\Delta$ at a cusp in the special direction ($t=0$), the surface has an $A_1$ singularity given locally by $y^2+x^3+sx=y^2+x(x^2+s)$.\\
\end{itemize}
These possibilities can be phrased more concisely by saying that $S=\pi^{-1}(L)$ is singular if and only if $L\in \Delta^*$, the dual curve of the discriminant $\Delta$.  The Pl\"{u}cker formulas allow us to compute $d^*=\deg(\Delta^*)$, $c^* = |\Delta^*_c|$, and $n^* = |\Delta^*_n|$ from the known values of $d=\deg(\Delta)$, $c=|\Delta_c|$ and $n=|\Delta_n|$.
\begin{align*}
d^* &= d(d-1) - 2n- 3c =612\\
c^* &= 3d(d-2) -6n -8c = 1944\\
n^* &= \frac{ d - d^*(d^*-1) + 3c^*}{2}=184032.
\end{align*}
where the third equation comes from dualizing the first.  Rephrasing the previous bullet points in terms of the dual curve picture, we have:
\begin{itemize}
\item If $L\in \Delta^*$ is general, the surface has an $A_1$ singularity. 
\item If $L\in \Delta^*$ is a flex, the surface has an $A_1$ singularity ($L$ meets $\Delta$ at a cusp in the special direction).
\item If $L\in \Delta^*$ is a node, the surface has two disjoint $A_1$ singularities ($L$ is a bitangent of $\Delta$).  
\item If $L\in \Delta^*$ is a cusp, the surface has an $A_2$ singularity ($L$ is tangent to $\Delta$ at a flex).\\
\end{itemize}
These singular surfaces have minimal resolutions $\widetilde{S}$, whose singular fibers fall into Kodaira's classification.  The $A_1$ singularity requires a single blow up, and the $A_2$ singularity requires two.  The possibilities for the Kodaira fibers and the polarized lattices of vertical curves are as follows:
\begin{itemize}
\item If $L\in \Delta^*$ is general, the resolved surface has a type $I_2$ (banana) fiber.\\
$A_0(\widetilde{S})$ is isomorphic to the $A_1$ root lattice.
\item If $L\in \Delta^*$ is a flex, the resolved surface has a type $III$ (tacnode) fiber.\\
$A_0(\widetilde{S})$ is isomorphic to the $A_1$ root lattice.
\item If $L\in \Delta^*$ is a node, the resolved surface has two type $I_2$ (banana) fibers.
$A_0(\widetilde{S})$ is isomorphic to the $A_1\times A_1$ root lattice.
\item If $L\in \Delta^*$ is a cusp, the resolved surface has a type $I_3$ (triangle) fiber.
$A_0(\widetilde{S})$ is isomorphic to the $A_2$ root lattice.\\
\end{itemize}
Torsion in the Mordell-Weil group of $S$ can be bounded in terms of the lattice $A_0(\widetilde{S})$, see \cite{mir}.  In our case, these lattices are sufficiently small that torsion never occurs.

\newpage
\nocite{*}
\bibliography{JBDraft}
\bibliographystyle{plain}

\end{document}